\documentclass[12pt]{amsart}
\usepackage{amsmath,latexsym,amsfonts,amssymb,amsthm}
\usepackage{geometry}
\usepackage{hyperref}
\usepackage{mathrsfs}
\usepackage{graphicx,color}
\usepackage{tikz-cd}
\usepackage{tikz}
\usetikzlibrary{positioning}
\usepackage{mathtools}
\usepackage[all,cmtip]{xy}

\numberwithin{equation}{section}

\newtheorem{prop}{Proposition}[section]
\newtheorem{thm}[prop]{Theorem}
\newtheorem{cor}[prop]{Corollary}

\newtheorem{lem}[prop]{Lemma}

\theoremstyle{definition}

\newtheorem{defn}[prop]{Definition}
\newtheorem{expl}[prop]{Example}
\newtheorem{rem}[prop]{\it Remark}

\newtheorem{ass}[prop]{Assumption}

\newtheorem*{claim*}{Claim}

\newcommand{\bP}{\mathbb{P}}

\newcommand{\bR}{\mathbb{R}}

\newcommand{\bQ}{\mathbb{Q}}
\newcommand{\bZ}{\mathbb{Z}}

\newcommand{\tF}{\widetilde{F}}
\newcommand{\tD}{\widetilde{D}}

\newcommand{\cX}{\mathcal{X}}

\newcommand{\cO}{\mathcal{O}}

\newcommand{\cI}{\mathcal{I}}

\newcommand{\cJ}{\mathcal{J}}

\newcommand{\Supp}{\mathrm{Supp}~}

\newcommand{\mult}{\mathrm{mult}}

\newcommand{\vol}{\mathrm{vol}}

\title[Weak Boundedness of Fano threefolds with large Seshadri constants]{Weak boundedness of Fano threefolds with large Seshadri constants in characteristic $p>5$}
\author{Ziquan Zhuang}
\address{Department of Mathematics, Princeton University, Princeton, NJ, 08544-1000.}
\email{zzhuang@math.princeton.edu}
\date{}

\begin{document}

\maketitle

\begin{abstract}
Given $\epsilon>0$, we show that over an algebraically closed field of characteristic $p>5$, the anticanonical volume of a Fano threefold $X$ (with arbitrary singularities) whose anticanonical divisor has Seshadri constant $\epsilon(-K_X,x)>2+\epsilon$ at some smooth point $x\in X$ is bounded from above.
\end{abstract}

\section{Introduction}

In this paper, we continue the investigation of Fano varieties with large anticanonical Seshadri constants in positive characteristic, as initiated in \cite{takumi,Pn_char_p}. First let us recall the definition of Seshadri constants.

\begin{defn}
Let $L$ be an ample $\bQ$-Cartier $\bQ$-divisor on a normal projective variety $X$ and $x\in X$ a smooth point. The \textit{Seshadri constant} of $L$ at $x$ is defined as
\[
  \epsilon(L,x):=\sup\{t\in\bR_{>0}\mid \sigma^*L-tE\textrm{ is nef}\},
\]
where $\sigma:\mathrm{Bl}_x X\to X$ is the blow-up of $X$ at $x$, and $E$ is the exceptional divisor of $\sigma$.
\end{defn}

When $X$ is a Fano variety, i.e. $-K_X$ is $\bQ$-Cartier and ample, we are interested in the Seshadri constants of its anticanonical divisor, which measure the local positivity of the Fano variety. It has been observed that over a field of characteristic zero, this local invariant, when it's large, also governs the global geometry of the Fano variety. More precisely, Fano varieties with large Seshadri constants enjoy nice geometric properties and satisfy certain boundedness:

\begin{thm} \label{thm:char 0} \cite{bs,lz16,LargeSeshadri}
Let $X$ be a complex Fano variety of dimension $n$ and let $x\in X$ be a smooth point. Let $\epsilon>0$ be a constant.
    \begin{enumerate}
        \item If $\epsilon(-K_X,x)>n$, then $X\cong\bP^n$;
        \item If $\epsilon(-K_X,x)\ge n$, then $X$ has klt singularities;
        \item If $\epsilon(-K_X,x)>n-1$, then $X$ is rationally connected;
        \item The set of those $X$ with $\epsilon(-K_X,x)>n-1+\epsilon$ for some $x\in X$ is weakly bounded;
        \item The set of those $X$ with $\epsilon(-K_X,x)\ge n-1$ for some $x\in X$ is birationally bounded.
    \end{enumerate}
\end{thm}

Here a set of Fano varieties is said to be weakly bounded if there exists a constant $M>0$ such that for any Fano variety $X$ in the set we have $\vol(-K_X)=((-K_X)^n)<M$. It is said to be birationally bounded if there exists a family of varieties $\pi:\cX\rightarrow B$ over a base of finite type such that any $X$ in the set is birational to $\cX_b=\pi^{-1}(b)$ for some $b\in B$.

The proof of this theorem relies on various consequences of the Kawamata-Viehweg vanishing theorem, so does not extend to positive characteristic. Still, we believe that the same statements also hold over a field of characteristic $p>0$ (possibly excluding a few small primes $p$). Indeed, it is proved by Murayama \cite[Theorem B]{takumi} that if $X$ is a smooth Fano variety (over an algebraically closed field of any characteristic) of dimension $n$ and $\epsilon(-K_X,x)\ge n+1$ for some $x\in X$ then $X\cong\bP^n$. This is later generalized by the author to the case of Fano variety with arbitrary singularities and $\epsilon(-K_X,x)>n$ in \cite{Pn_char_p}. We also know that $X$ is globally $F$-regular (hence in particular has klt singularities) when $\epsilon(-K_X,x)\ge n$ for some $x\in X$ and $p>2$ by the detailed classification in \cite{Pn_char_p}.

In the remaining part of the paper, let $k$ be an algebraically closed field of charateristic $p$. The purpose of this note is to generalize the weak boundedness statement (i.e. Theorem \ref{thm:char 0}(4)), in the case of Fano threefolds, to positive characteristic. Here is our main result.

\begin{thm} \label{thm:weakbdd}
Given $\epsilon>0$ and assume that $\mathrm{char}(k)=p>5$, then the set of Fano threefolds $X$ such that $\epsilon(-K_{X},x)>2+\epsilon$ for some smooth point $x\in X$ is weakly bounded.
\end{thm}

By \cite[Example 1.3]{LargeSeshadri}, it is not hard to see that our assumption on the Seshadri constants here is sharp. Combining with the main result of \cite{F-seshadri}, we have the following immediate corollary.

\begin{cor} \label{cor:birbdd}
Given $\epsilon>0$ and assume that $\mathrm{char}(k)=p>5$, then the set of Fano threefolds $X$ such that $\epsilon(-K_{X},x)>2+\epsilon$ for some smooth point $x\in X$ is birationally bounded.
\end{cor}

In view of \cite{LargeSeshadri}, it is better to consider Theorem \ref{thm:weakbdd} in the context of moving Seshadri constants (see Definition \ref{defn:moving Seshadri}), which generalize the notion of Seshadri constants to arbitrary divisors (not necessarily ample). It has the advantage of behaving better under typical operations in the Minimal Model Program (MMP). By a result of Demailly, the moving Seshadri constants and the original Seshadri constants coincide for ample divisors, hence Theorem \ref{thm:weakbdd} can be regarded as a special case of the following more general statement:

\begin{thm} \label{thm:weakbdd_m}
Given $\epsilon>0$ and assume that $\mathrm{char}(k)=p>5$, then there exists a constant $M$ depending only on $\epsilon$ such that if $X$ is a threefold with $\epsilon_{m}(-K_{X},x)>2+\epsilon$ for some smooth point $x\in X$, then $\vol(-K_X)<M$.
\end{thm}

In characteristic zero, the corresponding result \cite[Theorem 3.6]{LargeSeshadri} is proved using the Koll\'ar-Shokurov connectedness principle, whose validity remains open in positive characteristic (even for threefolds). Our proof strategy then, is to come up with several weaker versions of the Koll\'ar-Shokurov connectedness principle in positive characteristic, and construct a pair $(X,D)$ that violates one of these. On the other hand, by the work of \cite{mmp-hx,mmp-birkar,mmp-bw}, we may run the MMP for threefolds in characteristic $p>5$, and this easily reduces Theorem \ref{thm:weakbdd_m} to the case of Mori fiber spaces. Depending on the behaviour of the Mori fiber spaces, we have two cases to consider. If the Mori fiber space is of fiber type, we can prove that it is birational to a $\bP^2$-bundle over a curve and it is then relatively straightforward to find the desired boundary divisor $D$ using an argument similar to \cite{jiang}. The other case is when we have a terminal Fano variety of Picard number one. In this case, our argument relies heavily on the method of \cite{bpf} and the upshot is to find some highly singular divisors in the pluri-anticanonical system that cut out a zero dimension subscheme (which can be viewed as an analog of isolated log canonical center in characteristic zero) supported at a very general point. A key tool here is a local version (see Lemma \ref{lem:global generation}) of \cite[Theorem A]{non-nef}, which allows us to construct new singular divisors out of existing ones.

This paper is organized as follows. In Section \ref{sec:prelim}, we collect some definitions and useful results on $F$-singularity and moving Seshadri constants. We then use these results to reduce Theorem \ref{thm:weakbdd_m} to the case of Mori fiber spaces. In Section \ref{sec:fiber type}, we treat the case of strict Mori fiber spaces while in Section \ref{sec:rho=1}, we deal with the case of terminal Fano varieties of Picard number one.

\subsection*{Acknowledgement}
The author would like to thank his advisor J\'anos Koll\'ar for constant support, encouragement and numerous inspiring conversations. He also wishes to thank Yuchen Liu for helpful discussions.

\section{Preliminary} \label{sec:prelim}

\subsection{Notation and conventions}

Unless otherwise specified, all varieties are assumed to be normal and defined over an algebraically closed field $k$ of characteristic $p\ge0$. A \emph{pair} $(X,D)$ consists of a variety $X$ and an effective $\bQ$-divisor $D$ on $X$. A dominant morphism $f:X\rightarrow Y$ is called a \emph{fiber type} morphism if it has connected fibers and $0<\dim Y<\dim X$.

\subsection{$F$-singularities and test ideals} \label{sec:test_ideal}

We first recall a few definitions and results on singularities in characteristic $p$. All the statements in this section have appeared in the literature before, and we refer to \cite{ss} for more details on $F$-singularities and to \cite{test_ideal} for a general treatment of test ideals.

\begin{defn}
Let $X$ be a normal quasi-projective variety and $\Delta$ an effective $\bQ$-divisor on $X$. Fix a closed point $x\in X$. The pair $(X,\Delta)$ is called \emph{globally $F$-regular} (resp. \emph{globally sharply $F$-split}) if for all effective Weil divisor $D$ on $X$ (resp. for $D=0$), there exists an $e$ such that the composition
\begin{equation}\label{eq:f-reg}
\cO_X\rightarrow F_*^e\cO_X\hookrightarrow F_*^e\cO_X(\left\lceil (p^e-1)\Delta \right\rceil +D)
\end{equation}
splits as a map of $\cO_X$-modules. It is said to be \emph{strongly $F$-regular} (resp. \emph{sharply $F$-pure}) at $x$ if the pair is globally $F$-regular (resp. globally sharply $F$-split) in some affine neighbourhood of $x$.
\end{defn}

\begin{defn}
Let $L$ be an ample $\bQ$-divisor on $X$. Assume that $X$ is strongly $F$-regular, then the \emph{$F$-pure threshold} of the (polarized) pair $(X,L)$ is defined to be the supremum of all $t\ge0$ such that $(X,tD)$ is strongly $F$-regular for all effective $\bQ$-divisor $D\sim_\bQ L$. When $X$ is Fano, we define its $F$-pure threshold, denoted by fpt($X$), as the $F$-pure threshold of $(X,-K_X)$.
\end{defn}

By \cite[Lemma 3.4]{hw}, for any effective divisor $D$ on a normal variety $X$ and for any integer $e>0$, we have an isomorphism of $F_*^e\cO_X$-modules
\[\mathcal{H}om_{\cO_X}(F_*^e(\cO_X(D)),\cO_X)\cong F_*^e(\cO_X((1-p^e)K_X-D)).
\]
Viewing an element of $F_*^e(\cO_X((1-p^e)K_X-D))$ as a map $\theta:F_*^e(\cO_X(D))\rightarrow \cO_X$ and evaluating at $1\in F_*^e\cO_X \subseteq F_*^e(\cO_X(D))$, we obtain the trace map
\[\mathrm{Tr}_{X}^{e}(D):F_*^e(\cO_X((1-p^e)K_X-D))\rightarrow \cO_X.
\]
By abuse of notation, we will often denote $\mathrm{Tr}_{X}^{e}(D)$ simply by $\mathrm{Tr}_{X}^{e}$ or $\mathrm{Tr}^{e}$. It is quite straightforward to see that a pair $(X,\Delta)$ is sharply $F$-pure if and only if
\[\mathrm{Tr}_{X}^{e}:F_*^e(\cO_X(\lfloor (1-p^e)(K_X+\Delta)\rfloor))\rightarrow \cO_X
\]
is surjective for some $e>0$ (see e.g. \cite[Proposition 2.5]{bpf}).

Next we review the definition of test ideals on a smooth variety (this will be the only case we need). Our definition is taken from \cite{bms}. Let $X$ be a smooth variety and $\mathfrak{a}$ be an ideal sheaf on $X$. Let $e$ be a positive integer. Let $\mathfrak{a}^{[1/p^e]}$ denote the unique smallest ideal sheaf $\cJ$ such that $\mathfrak{a}\subseteq\cJ^{[p^e]}$. One can show that (see \cite[Proposition 2.5 and Lemma 2.8]{bms})
\[\mathrm{Tr}^e_X(F^e_*(\mathfrak{a}\cdot \omega_X))=\mathfrak{a}^{[1/p^e]}\cdot \omega_X
\]
and that if $c\ge 0$ then
\[(\mathfrak{a}^{\lceil cp^e \rceil})^{[1/p^e]}\subseteq (\mathfrak{a}^{\lceil cp^{e+1} \rceil})^{[1/p^{e+1}]}.
\]

\begin{defn}
Given $\mathfrak{a}$ and $c\ge0$ as above. The \emph{test ideal} of the pair $(X,\mathfrak{a}^c)$ is defined to be
\[\tau(X,\mathfrak{a}^c)=\bigcup_{e\ge0} (\mathfrak{a}^{\lceil cp^e \rceil})^{[1/p^e]}.
\]
\end{defn}

Since $\cO_X$ is Noetherian, this is well-defined and by the previous discussion we have \[\tau(X,\mathfrak{a}^c)\cdot \omega_X=\mathrm{Tr}^e_X(F^e_*(\mathfrak{a}^{\lceil cp^e \rceil}\cdot \omega_X))
\]
for $e\gg0$. If $\mathfrak{a}=\cO_X(-D)$ for some divisor $D\ge0$, we simply write $\tau(X,\mathfrak{a}^c)$ as $\tau(X,\Delta)$ where $\Delta=cD$.

The following property of test ideals is well known to expert. It is analogous to the corresponding property of multiplier ideals in characteristic zero.

\begin{lem} \label{lem:mult and test ideal}
Let $X$ be a smooth variety of dimension $n$. Let $x\in X$ and let $D$ be an effective $\bQ$-divisor on $X$.
\begin{enumerate}
    \item If $\mult_x D<1$, then the pair $(X,D)$ is strongly $F$-regular at $x$. In particular, $\tau(X,D)_x=\cO_{X,x}$.
    \item If $\mult_x D=N\ge n$, then $\tau(X,D)\subseteq\mathfrak{m}_x^{\lfloor N-n+1 \rfloor}$. In particular, $(X,D)$ is not strongly $F$-regular at $x$.
\end{enumerate}
\end{lem}

\begin{proof}
(2) is a consequence of \cite[Proposition 3.3]{non-nef}, so we only need to prove (1). If $\cO_X(-\lceil p^e D \rceil)\subseteq \mathfrak{m}_x^{[p^e]}$ for $e\gg0$ then $\mult_x D\ge 1$, hence if $\mult_x D<1$ then $\tau(X,D)_x=\cO_{X,x}$ by definition. If follows that $(X,(1+\epsilon)D)$ is sharply $F$-pure for some $0<\epsilon\ll 1$. By \cite[Theorem 3.9]{ss}, this implies that $(X,D)$ is strongly $F$-regular at $x$.
\end{proof}

\begin{cor} \label{cor:fpt-Pn}
Let $X=\bP^n$ and $D$ an effective $\bQ$-divisor of degree $<1$ on $X$. Then $(X,D)$ is globally $F$-regular.
\end{cor}

\begin{proof}
This follows from Lemma \ref{lem:mult and test ideal}(1) by taking the cone over $(X,D)$.
\end{proof}

\subsection{Moving Seshadri constants}

We now recall some results on moving Seshadri constants and carry out the reduction step of Theorem \ref{thm:weakbdd_m}. Essentially all the results in this section are contained in \cite{LargeSeshadri} as the same proofs there work in any characteristic.

\begin{defn} \label{defn:moving Seshadri} \cite{nakamaye,elmn}
Let $D$ be a $\bQ$-divisor on a normal variety $X$. Let $x\in X$ be a smooth point. The moving Seshadri constant of $D$ at $x$ is defined to be 
\[\epsilon_m(D,x)=\limsup_k\frac{s(kD,x)}{k}
\]
where the limit is taken over sufficiently large and divisible integers $k$ and $s(kD,x)$ is the largest integer $s\ge -1$ such that the natural map
\[H^{0}(X,\cO_X(kD))\rightarrow H^{0}(X,\cO_X(kD)\otimes\cO_X/\mathfrak{m}_{x}^{s+1})
\]
is surjective. We also define
\[\epsilon_m(D)=\sup_{x\in X^{\circ}}\epsilon_m(D,x)
\]
where $X^\circ$ is the smooth locus of $X$. It is not hard to see that the supremum is actually a maximum and is achieved at a very general point of $X$.
\end{defn}

Here are some of the basic properties of moving Seshadri constants.

\begin{lem} \label{lem:restrict_m} 
Let $L$ be a $\bQ$-Cartier $\bQ$-divisor on $X$. Let $Y$ be a positive dimensional subvariety of $X$ and $x$ a smooth point of both $X$ and $Y$. Then $\epsilon_m(L,x)\le\epsilon_m(L|_Y,x)$.
\end{lem}

\begin{proof}
See \cite[Lemma 3.1]{LargeSeshadri}.
\end{proof}

\begin{lem} \label{lem:nondecrease}
Let $\phi:X\dashrightarrow Y$ be a birational contraction between normal varieties, then $\epsilon_{m}(-K_{X})\le\epsilon_{m}(-K_{Y})$ and $\vol(-K_X)\le\vol(-K_Y)$.
\end{lem}

\begin{proof}
The first statement is simply \cite[Corollary 3.3]{LargeSeshadri} while the second follows from the injection $H^0(X,-mK_X)\rightarrow H^0(Y,-mK_Y)$ induced by $\phi$.
\end{proof}

\begin{defn} \cite[Theorem 1.33]{mmp}
A projective birational morphism $\phi:Y\rightarrow X$ is called a \emph{terminal modification} of $X$ if $Y$ is $\bQ$-factorial, terminal and $K_Y$ is $\phi$-nef.
\end{defn}

The existence of terminal modification is a formal consequence of the MMP, so by the work of \cite{mmp-hx,mmp-birkar,mmp-bw}, terminal modification exists for threeholds in characteristic $p>5$.

\begin{lem} \label{lem:mmodel}
Let $\phi:Y\rightarrow X$ be a terminal modification of $X$, then $\epsilon_m(-K_X)=\epsilon_m(-K_Y)$ and $\vol(-K_X)=\vol(-K_Y)$.
\end{lem}

\begin{proof}
The first equality follows from \cite[Lemma 3.4]{LargeSeshadri}. To see the second equality, let $D_X\in|-mK_X|$ and $D_Y$ its strict transform on $Y$. We may write $mK_Y+D_Y+E_Y\sim\phi^*(mK_X+D_X)\sim 0$ for some $\phi$-exceptional divisor $E_Y$. Note that $E_Y$ has integral coefficients. Apply \cite[Lemma 2.5]{LargeSeshadri} to the pair $(X,\frac{1}{m}D)$ we see that $E_Y$ is effective. It follows that $D_Y+E_Y\in|-mK_Y|$ and we have an injection $\phi^{-1}_*:H^0(X,-mK_X)\rightarrow H^0(Y,-mK_Y)$. On the other hand $\phi_*$ also induces an inclusion $H^0(Y,-mK_Y)\rightarrow H^0(X,-mK_X)$ and $\phi_*\circ\phi^{-1}_*=\mathrm{id}$, hence it's an isomorphism and $\vol(-K_X)=\vol(-K_Y)$.
\end{proof}

Using these properties we can easily reduce the proof of Theorem \ref{thm:weakbdd_m} to the case of Mori fiber spaces.

\begin{defn}
Let $X$ be a normal variety and $f:X\rightarrow Y$ a projective morphism with $f_*\cO_X=\cO_Y$. Then $f$ is called a \emph{Mori fiber space} if
\begin{enumerate}
    \item $X$ has $\bQ$-factorial terminal singularities,
    \item the relative Picard number $\rho(X/Y)=1$, and
    \item $-K_X$ is $f$-ample.
\end{enumerate}
\end{defn}

\begin{lem} \label{lem:reduction}
It suffices to prove Theorem \ref{thm:weakbdd_m} when the threefold $X$ admits a Mori fiber space structure.
\end{lem}

\begin{proof}
Let $Y\rightarrow X$ be a terminal modification of $X$ and run the $K_Y$-MMP on $Y$. Since $K_X$ is not pseudoeffective by assumption, the MMP ends with $Y\dashrightarrow Y_1$ where $Y_1$ admits a Mori fiber space structure. By Lemma \ref{lem:nondecrease} and \ref{lem:mmodel}, we have $\epsilon_m(-K_{Y_1})\ge\epsilon_m(-K_X)$ and $\vol(-K_{Y_1})\ge\vol(-K_X)$. Thus if Theorem \ref{thm:weakbdd_m} holds for Mori fiber spaces then it holds for all threefolds as well.
\end{proof}

\section{Fiber type case} \label{sec:fiber type}

We first look at the case when the Mori fiber space $f:X\rightarrow Y$ is of fiber type, which will be assumed for all Mori fiber spaces is this section. Consider the following assumptions on a variety $X$ of dimension $n$:

\begin{ass} \label{ass:raylength}
There exists a constant $A=A(n,\epsilon)>0$ that only depends on $n$ and $\epsilon$ such that every $K_X$-\emph{positive} extremal ray of $X$ is generated by a curve $C$ with $(K_X\cdot C)<A$.
\end{ass}

\begin{ass} \label{ass:seshadri_bound}
$\epsilon_{m}(-K_{X})>n-1+\epsilon$.
\end{ass}

Our plan in this section is to show that the set of Mori fiber spaces $f:X\rightarrow Y$ that satisfy these two assumptions is weakly bounded. Using suitable cone theorem for threefold pairs in positive characteristics, we find that \ref{ass:raylength} is implied by \ref{ass:seshadri_bound} in dimension 3, thus proving Theorem \ref{thm:weakbdd_m} in the case of Mori fiber spaces.

We first deduce some direct consequences of \ref{ass:raylength} and \ref{ass:seshadri_bound}.

\begin{lem} \label{lem:F and Y}
Let $f:X\rightarrow Y$ be a Mori fiber space such that $X$ satisfies \ref{ass:seshadri_bound}, then the general fiber $F$ of $f$ is reduced and isomorphic to $\bP^{n-1}$, $\dim Y=1$ and $\rho(X)=2$.
\end{lem}

\begin{proof}
Let $\tF$ be the normalization of the reduced subscheme of $F$. By \cite[Theorem 1.1]{general_fiber}, there exists an effective Weil divisor $D$ on $\tF$ such that $K_{\tF}+D\sim K_X|_{\tF}$. Since $f:X\rightarrow Y$ is a Mori fiber space, $-K_X|_{\tF}$ is ample, hence by \ref{ass:seshadri_bound} and Lemma \ref{lem:restrict_m}, we have $\epsilon(-K_{\tF}-D)\ge\epsilon_m(-K_X)>n-1\ge \dim \tF$. By \cite[Theorem 3]{Pn_char_p}, this implies that $\tF\cong\bP^{n-1}$, $\dim Y=1$ and $D=0$ (note that components of $D$ have integral coefficients). But then by \cite[Theorem 7.1, Lemma 7.2]{Badescu}, $F$ itself is reduced and we also have $(K_X+F)|_{\tF}=K_{\tF}$. Since $\tF$ is smooth, $F$ is normal by \cite[Theorem A]{das}, thus $F\cong\tF\cong\bP^{n-1}$. Finally, since $f$ is a Mori fiber space, we get $\rho(X)=\rho(X/Y)+\rho(Y)=2$.
\end{proof}

\begin{lem} \label{lem:f-reg along F}
Let $f:X\rightarrow Y$ be a Mori fiber space with general fiber $F$. Assume that $X$ satisfies \ref{ass:seshadri_bound}, then there exists $D\sim_\bQ -\frac{1}{1-\epsilon}K_X$ such that $(X,D)$ is strongly $F$-regular along $F$.
\end{lem}

\begin{proof}
This follows from Lemma \ref{lem:mult and test ideal} and the same proof of \cite[Corollary 5.2]{LargeSeshadri}.
\end{proof}

\begin{lem} \label{lem:explicit_ample}
Let $f:X\rightarrow Y$ be a Mori fiber space with general fiber $F$. Assume that $X$ satisfies both \ref{ass:raylength} and \ref{ass:seshadri_bound}. Then $-K_X+AF$ is ample. 
\end{lem}

\begin{proof}
We may assume that $-K_X$ is not nef. By Lemma \ref{lem:F and Y}, $F$ is smooth divisor in $X$, $Y$ is a curve and $\rho(X)=2$. It follows that the Mori cone $\overline{NE}(X)$ is 2-dimensional, hence by \ref{ass:raylength}, it is generated by a curve $C$ that dominates $Y$ with $0<(K_X\cdot C)<A$ and a line $l$ in $F$. Let $L=-K_X+AF$, then since $F$ is Cartier we have $(L\cdot C)\ge (-K_X\cdot C)+A>0$ and it is clear that $(L\cdot l)>0$. Hence $L$ lies in the interior of the nef cone and is therefore ample.
\end{proof}

We will also need the following result:

\begin{lem} \label{lem:Qeff}
Let $\phi:X\rightarrow Y$ be a projective morphism onto a curve $Y$ with general fiber $F$. Let $L$ be a $\bQ$-Cartier big divisor on $X$ and let $\lambda<\frac{\vol_X(L)}{n\cdot \vol_F(L|_F)}$ be a rational number where $n=\dim X$, then $L-\lambda F$ is $\mathbb{Q}$-effective.
\end{lem}

\begin{proof}
Let $m$ be a sufficiently divisible positive integer and let $k=m\lambda$.
We may assume $\lambda>0$ and $k\in\mathbb{Z}$. By the exact sequence
\[
0\rightarrow\mathcal{O}_{X}(mL-lF)\rightarrow\mathcal{O}_{X}(mL-(l-1)F)\rightarrow\mathcal{O}_{F}(mL)\rightarrow0
\]
for $l=1,\cdots,k$ we get 
\begin{eqnarray*}
h^{0}(X,\mathcal{O}_{X}(mL-kF)) & \ge & h^{0}(X,\mathcal{O}_{X}(mL))-k\cdot h^{0}(F,\mathcal{O}_{F}(mL)) \\
                                & =   & [\frac{\vol_X(L)}{n!}-\lambda\frac{\vol_F(L|_F)}{(n-1)!}]m^{n}+O(m^{n-1})>0, 
\end{eqnarray*}
where the middle equality is the asymptotic Riemann-Roch formula and the last inequality follows from the assumption on $\lambda$. Since $mL-kF=m(L-\lambda F)$, the lemma follows.
\end{proof}

\begin{rem}
When $F$ is normal, the above lemma also holds even if $L$ is not $\bQ$-Cartier since in this case the restriction $L|_F$ is well defined (it is determined by the restriction of $L$ to the smooth locus of $F$). Also note that if in addition $L$ has integral coefficients, then in the above argument $m$ can be taken as any sufficiently large integer such that $m\lambda\in\bZ$. Therefore, if the denominator of $\lambda$ is not divisible by $p$, then we can choose $\Delta\sim_\bQ L-\lambda F$ such that $m\Delta$ has integral coefficient for some integer $m$ that is not divisible by $p$.
\end{rem}

Recall that in characteristic zero, the weak boundedness of varieties with large moving Seshadri constants \cite{LargeSeshadri} is proved using the connectedness lemma of Koll\'ar-Shokurov, which is not yet available in positive characteristic. The following weaker version, however, suffices for the purpose of this section.

\begin{lem} \label{lem:weak connectedness}
Let $(X,D)$ be a pair such that $X$ is projective and $-(K_X+D)$ is ample, then $(X,D)$ has at most one {\it good} $F$-pure center.
\end{lem}

Let us elaborate the meaning of {\it good $F$-pure center} here. Let $(X,D)$ be a pair and $Y\subseteq X$ a normal subvariety such that the following conditions hold in a neighbourhood of $Y$ (we refer to \cite{schwede-adjunction} for the definition of center of F-purity):
\begin{enumerate}
    \item The Cartier index of $K_X+D$ is not divisible by $p$;
    \item $(X,D)$ is sharply $F$-pure along $Y$ and $Y$ is a center of sharp $F$-purity for $(X,D)$.
\end{enumerate}
Then by the main theorem of \cite{schwede-adjunction}, there exists a canonically determined effective divisor $D_Y$ such that $(K_X+D)|_Y\sim_\bQ K_Y+D_Y$. We say that $Y$ is a good $F$-pure center of $(X,D)$ if $(Y,D_Y)$ is globally $F$-regular.

\begin{proof}
Suppose there are two distinct good $F$-pure centers $W_1$, $W_2$ of the pair $(X,D)$ and we will derive a contradiction. By the main theorem of \cite{schwede-adjunction}, both $W_i$ are minimal among centers of sharp $F$-purity for $(X,D)$, hence by \cite[Lemma 3.5]{schwede-F-center}, $W_1$ is disjoint from $W_2$. Let $W=W_1\cup W_2$ and let $e>0$ be a sufficiently divisible integer. We have the following commutative diagram
\[
\xymatrix{F_{*}^{e}\mathcal{O}_{X}((1-p^{e})(K_{X}+D))\ar[r]^{\psi}\ar[d]^{\mathrm{Tr}_{X}^{e}} & F_{*}^{e}\mathcal{O}_{W}((1-p^{e})(K_{W}+D_{W}))\ar[d]^{\mathrm{Tr}_{W}^{e}}\\
\mathcal{O}_{X}\ar[r]^{\phi} & \mathcal{O}_{W}.
}
\]
Since the $W_i$'s are globally $F$-regular, $H^0(\mathrm{Tr}_{W}^{e})$ is surjective. On the other hand, the cokernel of $H^0(\psi)$ lies in $H^1(X,\mathcal{I}_{W}((1-p^{e})(K_{X}+D))$, which vanishes for $e\gg 0$ since $-(K_X+D)$ is ample. Hence $\phi\circ\mathrm{Tr}_{X}^{e}=\mathrm{Tr}_{W}^{e}\circ\psi$ induces a surjection on $H^0$. In particular, the natural restriction $H^0(X,\mathcal{O}_{X})\rightarrow H^0(W,\mathcal{O}_{W})$ is surjective. But as $W$ contains two connected components, this is a contradiction.
\end{proof}

We now prove that the assumptions \ref{ass:raylength} and \ref{ass:seshadri_bound} together imply weak boundedness for Mori fiber spaces.

\begin{thm} \label{thm:mfs}
Given $v,\alpha>0$, there exists a constant $M=M(n,A,v,\alpha)$ depending only on $n$, $A$, $v$ and $\alpha$ such that if $f:X\rightarrow Y$ is a Mori fiber space such that $X$ satisfies \ref{ass:raylength}, $Y$ is a curve and the general fiber $F$ is globally $F$-regular with $\vol(-K_F)<v$ and  $\mathrm{fpt}(F)>\alpha$, then $\vol(-K_X)<M$.
\end{thm}

\begin{proof}
We may assume $\alpha<1$. Let $0<\lambda<(nv)^{-1}\vol(-K_X)$, $0<r<\alpha$ be rational numbers whose denominators are not divisible by $p$. Apply Lemma \ref{lem:Qeff} and its subsequent remark to $L=-K_{X}$ we see that there exists an effective divisor $\Delta\sim_{\mathbb{Q}}-K_{X}-\lambda F$ such that $m\Delta$ has integral coefficients for some $p\nmid m$. Let $D=F_1+F_2+r\Delta$ where $F_1$ and $F_2$ are two distinct general fibers of $f$ . We have $-(K_{X}+D)\sim_{\mathbb{Q}}-(1-r)K_{X}+(r\lambda-2)F$. Suppose that $r\lambda-2\ge A(1-r)$ where $A$ is the constant in \ref{ass:raylength}, then $-(K_{X}+D)$ is ample by Lemma \ref{lem:explicit_ample}. Perturbing $r$, we may assume that $(1-r)lK_X$ is Cartier for some $p\nmid l$. It follows that the Cartier index of $K_X+D$ is not divisible by $p$. On the other hand, as $(K_X+D)|_{F_i}\sim K_{F_i}+D_{F_i}$ where $D_{F_i}\sim_{\bQ}-rK_{F_i}$ and $r<\mathrm{fpt}(F)$, we see that $(F_i,D_{F_i})$ is globally $F$-regular. By \cite[Theorem A]{das}, $(X,D)$ is purely $F$-regular along $F_i$ and it follows that both $F_i$ are good $F$-pure centers for $(X,D)$, which contradicts Lemma \ref{lem:weak connectedness}. Hence we always have $r\lambda-2<A(1-r)$ and since $\lambda$ (resp. $r$) can be arbitrarily close to $(nv)^{-1}\vol(-K_X)$ (resp. $\alpha$), it follows immediately that $\vol(-K_{X})$ is bounded from above by a constant $M(n,A,v,\alpha)$ depending only on $n$, $A$, $v$ and $\alpha$.
\end{proof}

\begin{cor} \label{cor:mfs}
There exists a constant $M=M(n,\epsilon)$ depending only on $n$ and $\epsilon$ such that if $f:X\rightarrow Y$ is a Mori fiber space such that $X$ satisfies \ref{ass:raylength} and \ref{ass:seshadri_bound}, then $\vol(-K_X)<M$.
\end{cor}

\begin{proof}
Let $F$ be the general fiber of $f$. By Lemma \ref{lem:F and Y}, $F\cong\mathbb{P}^{n-1}$ and $Y$ is a curve. By Corollary \ref{cor:fpt-Pn}, the existence of $M$ follows from Theorem \ref{thm:mfs} by taking any $v>n^{n-1}$ and $\alpha<1$. 
\end{proof}

In the remaining part of the section, we assume that $p>5$. We proceed to show that \ref{ass:seshadri_bound} implies \ref{ass:raylength} for Mori fiber spaces in dimension at most 3.

\begin{lem} \label{lem:raylength}
Let $(X,D)$ be a pair with $\dim X\le 3$ and $R$ a $(K_{X}+D)$-negative extremal ray. Assume that 
\begin{enumerate}
\item $R$ is generated by a curve; 
\item Every curve generating $R$ is not contained in the non-klt locus
of $(X,D)$.
\end{enumerate}
Then $R$ is generated by a rational curve $C$ such that $0<-(K_{X}+D\cdot C)\le2\dim X$.
\end{lem}

\begin{proof}
The proof is the same as that of \cite[Theorem 2.13]{jiang}, except that we use \cite[Theorem 1.2]{mmp-birkar} for the existence of log minimal model and \cite[Theorem 1.1]{mmp-bw} for cone theorem for klt pairs in dimension 3.
\end{proof}

\begin{lem} \label{lem:3-fold-mfs}
Let $f:X\rightarrow Y$ be a Mori fiber space such that $\dim X\le 3$. Assume that $X$ satisfies \ref{ass:seshadri_bound}. Then $X$ also satisfies \ref{ass:raylength}.
\end{lem}

\begin{proof}
By Lemma \ref{lem:F and Y}, $\rho(X)=2$. Let $l$ be the class of a line in the general fiber $F$ of $f$ and let $R$ be the other extremal ray of $\overline{NE}(X)$. By Lemma \ref{lem:f-reg along F}, there exists $D\sim_\bQ -\frac{1}{1-\epsilon}K_X$ such that $(X,D)$ is klt along $F$. We may assume that $X$ is not weak Fano, otherwise there is nothing to verify. In particular, $(-K_X\cdot l)>0$ while $(-K_X\cdot R)<0$. Since $K_X+D\sim_\bQ -\frac{\epsilon}{1-\epsilon}K_X$ is $\bQ$-effective and has negative intersection with $R$, we see that $R$ is generated by a curve on $X$ by \cite[Proposition 5.5.2]{keel}. By construction, the non-klt locus of $(X,D)$ is contained in some special fibers of $f$, hence since $R$ has positive intersection with $F$, it satisfies all the assumptions of Lemma \ref{lem:raylength}. It follows that $R$ is generated by a curve $C$ such that
\[
0<\frac{\epsilon}{1-\epsilon}(K_X\cdot C)=-(K_{X}+D\cdot C)\le2\dim X\le 6
\]
and thus $X$ satisfies \ref{ass:raylength} by taking $A=\frac{6(1-\epsilon)}{\epsilon}$.
\end{proof}

\section{Picard number one case} \label{sec:rho=1}

Now we consider the case of terminal threefolds of Picard number one. Of course in this case Theorem \ref{thm:weakbdd_m} is just a special case of  the weak BAB conjecture in positive characteristic. Unfortunately this conjecture is still open even in dimension three, so we need a somewhat different approach. Similar to the fiber type case, our strategy is to prove an appropriate version of the Koll\'ar-Shokurov connectedness principle in positive characteristic and then, under the assumption that $X$ has large anticanonical volume, construct a pair $(X,D)$ that violates this principle. We start by setting up the framework.

Let $(X,B)$ be a pair and $x\in X$ a closed point such that $(X,B)$ is strongly $F$-regular at $x$. Let $D_1,\cdots,D_r$ be divisors on $X$ whose set theoretic intersection $\cap D_i$ equals $\{x\}$ in a neighbourhood of $x$. Let $\mathbf{D}=(D_1,\cdots,D_r)$ and let $\Delta(\mathbf{D})\subseteq\bR^r$ be the closure of the set of all $r$-tuples $(t_1,\cdots,t_r)\in\bQ^r_{\ge 0}$ such that $(X,B+\sum_{i=1}^r t_i D_i)$ is sharply $F$-pure at $x$. It can be viewed as an analog of the log canonical threshold polytope (see e.g. \cite{lct-polytope}) in positive characteristic. Clearly $\Delta(\mathbf{D})$ is convex. Let $\succ$ be the lexicographic ordering on $\bR^r$, namely, $t\succ t'$ if and only if $t\neq t'$ and the first non-zero entry of $t-t'$ is positive. We may then talk about the {\it dominant vertex} of $\Delta(\mathbf{D})$, defined to be the unique point $v=(a_1,\cdots,a_r)\in\Delta(\mathbf{D})$ such that for all $v'\in\Delta(\mathbf{D})$ we have $v\succeq v'$. Let $\Gamma(\mathbf{D})=\sum_{i=1}^r a_i D_i$. Any divisor $\Gamma$ of this form (i.e. there exists $\mathbf{D}$ as above such that $\Gamma=\Gamma(\mathbf{D})$) will be called an {\it $F$-pure combination} with an isolated center at $x$. Intuitively, one may view $(X,B+\Gamma)$ as an analog of a pair with an isolated log canonical center at $x$.

We can also define successive approximations of $\Gamma(\mathbf{D})$ as follows (c.f. \cite[\S 3]{bpf}). Let $(X,B)$ and $\mathbf{D}$ be as before and $e>0$ a positive integer such that $(p^e-1)(K_X+B)$ has integral coefficients, we define the the associated {\it $F$-threshold functions}  $t_i(e)\,(i=1,\cdots,r)$ inductively by taking $t_{l+1}(e)$ to be the largest integer $m\ge0$ such that the trace map
\[\mathrm{Tr}^{e}:F_{*}^{e}(\mathcal{O}_{X}((1-p^{e})(K_{X}+B)-\sum_{i=1}^l t_i(e)D_i-mD))\rightarrow\mathcal{O}_{X}\]
is locally surjective around $x$. It is then clear that $\frac{1}{p^e-1}(t_1(e),\cdots,t_r(e))\in\Delta(\mathbf{D})$ and their limit as $e\rightarrow\infty$ is exactly the dominant vertex of $\Delta(\mathbf{D})$. Let $W$ be the scheme-theoretic intersection of all the $D_i$, then by construction for all $j=1,\cdots,r$, 
\[\mathrm{Tr}^{e}:F_{*}^{e}(\mathcal{O}_{X}((1-p^{e})(K_{X}+B)-\sum_{i=1}^r t_i(e)D_i-D_j))\rightarrow\mathcal{O}_{X}
\]
is not surjective around $x$, thus
\begin{equation} \label{eq:not_surj}
\mathrm{Tr}^{e}(F_{*}^{e}(\mathcal{O}_{X}((1-p^{e})(K_{X}+B)-\sum_{i=1}^r t_i(e)D_i)\cdot \mathcal{I}_W))\subseteq\mathfrak{m}_x.
\end{equation}

We can now state the connectedness result we will use in this section:

\begin{lem} \label{lem:weak connectedness-2}
Let $(X,B)$ be a pair such that $X$ is projective and $\bQ$-factorial and there exists an integer $e>0$ such that $(p^e-1)B$ has integral coefficients. Let $x$, $y$ be general points on $X$ and let $\Gamma_x$ \emph{(}resp. $\Gamma_y$\emph{)} be an $F$-pure combination with isolated center at $x$ \emph{(}resp. $y$\emph{)}. Then the divisor $-(K_X+B+\Gamma_x+\Gamma_y)$ is not ample.
\end{lem}

\begin{proof}
Since $x$, $y$ are general we may assume that they're smooth points. Let $\Gamma_x=\Gamma(\mathbf{D}_x)$ where $\mathbf{D}_x=(D_{1,x},\cdots,D_{r,x})$ and let $W_x=\cap_{i=1}^r D_{i,x}$. We have $\Supp(W_x)=\{x\}$. For sufficiently divisible integer $e>0$, let $t_{i,x}(e)$ be the $F$-threshold function associated to $\mathbf{D}_x$ at $x$ and let $\Gamma_x^{(e)}=\sum_{i=1}^r t_{i,x}(e)D_{i,x}$. Similarly we have corresponding objects indexed by $y$. Let $W=W_x\cup W_y$, then by (\ref{eq:not_surj}), we have
\[\mathrm{Tr}^{e}(F_{*}^{e}(\mathcal{O}_{X}((1-p^{e})(K_{X}+B)-\Gamma_x^{(e)}-\Gamma_y^{(e)})\cdot \mathcal{I}_W))\subseteq\mathfrak{m}_x\cdot \mathfrak{m}_y.
\] 
Hence for $L^{(e)}=\mathcal{O}_{X}((1-p^{e})(K_{X}+B)-\Gamma_x^{(e)}-\Gamma_y^{(e)})$ we have the following commutative diagram
\[\xymatrix{
0\ar[r] & F_{*}^{e}(\mathcal{I}_W\cdot L^{(e)})\ar[d]^{\mathrm{Tr}^{e}}\ar[r] & F_{*}^{e}L^{(e)}\ar[d]^{\mathrm{Tr}^{e}}\ar[r] & F_{*}^{e}(L^{(e)}|_W)\ar[d]^{\mathrm{Tr}^{e}}\ar[r] & 0 \\
0\ar[r] & \mathfrak{m}_x\cdot\mathfrak{m}_y\ar[r] & \cO_X\ar[r] & k_x\oplus k_y\ar[r] & 0.
}\]
By construction $\mathrm{Tr}^{e}:F_{*}^{e}L^{(e)}\rightarrow\cO_X$ is locally surjective around $x$ and $y$, thus the induced map $\mathrm{Tr}^{e}:F_{*}^{e}(L^{(e)}|_W)\rightarrow k_x\oplus k_y$ is also surjective. As $\dim W=0$, we get a surjection
\[H^0(W,F_{*}^{e}(L^{(e)}|_W))\twoheadrightarrow k_x\oplus k_y.
\]
On the other hand as $e$ goes to infinity $\frac{1}{p^e-1}\Gamma_x^{(e)}$ tends to $\Gamma_x$, thus if $-(K_X+B+\Gamma_x+\Gamma_y)$ is ample then by Fujita type vanishing we have 
\[H^1(X,F_{*}^{e}(\mathcal{I}_W\cdot L^{(e)}))=H^1(X,\mathcal{I}_W\cdot L^{(e)})=0,
\]
which implies that $H^0(X,F_{*}^{e}L^{(e)})\rightarrow H^0(W,F_{*}^{e}(L^{(e)}|_W))$ is also surjective. As in Lemma \ref{lem:weak connectedness}, we deduce that the natural restriction $H^0(X,\cO_X)\rightarrow k_x\oplus k_y$ is surjective, a contradiction. Hence $-(K_X+B+\Gamma_x+\Gamma_y)$ cannot be ample. 
\end{proof}

Let $X$ be a $\bQ$-factorial terminal Fano threefold of Picard number 1 such that $\epsilon(-K_X)>2+\epsilon$. Suppose that $\vol(-K_X)$ can be arbitrarily large. Then to derive a contradiction to Lemma \ref{lem:weak connectedness-2}, it suffices to find an $F$-pure combination $\Gamma$ with isolated center at a very general point $x\in X$ such that $\Gamma\sim_\bQ -\lambda K_X$ for some $\lambda<\frac{1}{2}$. The remaining part of this section will be devoted to the construction of such $\Gamma$. Roughly speaking, we will construct three divisors $D_1$, $D_2$, $D_3$ that are very singular at $x$ (so as to make $\Gamma$ small), and the main technical point is to cut down the dimension (at $x$) of their intersection (in general, singular divisors can be quite rigid and hard to deform).

The situation is very much like constructing isolated lc center in characteristic zero and it is always straightforward to come up with the first singular divisor. Let $x\in X$ be a very general point and let $N\in\bZ_{>0}$ be a sufficiently large constant that will be determined later. Suppose that $\vol(-K_X)>N^6$, then there exists an effective $\bQ$-divisor $D_1\sim_\bQ -K_X$ such that $\mult_x D_1>N^2$. Since $X$ has Picard number one, we may assume that $D_1$ is irreducible and write $D_1=t\Delta$ where $\Delta=\Supp(D_1)$. A priori the multiple $t$ can be large. Our first claim is that $t$ can be bounded in terms of $\epsilon$.

\begin{lem} \label{lem:multiple bounded}
Let $X$ be a $\bQ$-factorial Fano variety of Picard number one and dimension $n$. Assume that $X$ is smooth in codimension two and $\epsilon(-K_X)>n-1+\epsilon$. Then there exists a constant $a=a(n,\epsilon)$ depending only on $n$ and $\epsilon$ such that for all divisor $D\subseteq X$ passing through a general point we have $-K_X\sim_\bQ tD$ for some $t<a$.
\end{lem}

\begin{proof}
We may assume that $D$ is irreducible. Let $\nu:\tD\rightarrow D$ be the normalization of $D$ and let $\Delta$ be the conductor divisor on $\tD$. Since $X$ is smooth in codimension two, by adjunction we have $K_{\tD}+\Delta=\nu^*K_D=\nu^*(K_X+D)$. Let $x$ be a general point on $D$. By \cite[Theorem 3]{Pn_char_p}, we have $\epsilon(-K_X-D,x)\le n+1$, hence if $t>n+1$, then by the definition of Seshadri constant we obtain
\begin{equation} \label{eq:compare_Seshadri}
\epsilon(-K_{\tD}-\Delta,x)=\epsilon(-K_X-D,x)>(1-\frac{1}{t})(n-1+\epsilon). 
\end{equation}
So if $(1-\frac{1}{t})(n-1+\epsilon)>n-1$, then by \cite[Theorem 3]{Pn_char_p} again we get $\tD\cong\bP^{n-1}$ and $\Delta=0$. Substituting back to (\ref{eq:compare_Seshadri}) we see that $\epsilon(-K_X,x)>\epsilon(-K_X-D,x)=n$, which forces $X\cong\bP^n$ and contradicts $t>n+1$. It follows that we have either $t\le n+1$ or $(1-\frac{1}{t})(n-1+\epsilon)\le n-1$. In either case the existence of the constant $a(n,\epsilon)$ is clear.
\end{proof}

The following example shows that the assumptions in the lemma are necessary.

\begin{expl}
Let $X=\bP(1^2,d^{n-1})$ and let $H$ be the ample generator of $\mathrm{Cl}(X)$. Then for any point $x\in X$ there exists a divisor $D\sim H$ containing $x$ but $-K_X\sim (2+d(n-1))D$. Note that $\epsilon(-K_X)=n-1+\frac{2}{d}$.
\end{expl}

In the sequel, fix a constant $a=a(3,\epsilon)$ that satisfies the conclusion of the previous lemma. If $N\gg a$, then by the previous lemma the singularities of $D_1$ mainly come from the singularity of $\Delta$, which has codimension at least two. Our next claim is that if $N$ is sufficiently large, then we can find another divisor $D_2\sim_\bQ -K_X$, whose support does not contain $\Delta$ (so in particular, $\dim D_1\cap D_2 \le 1$), such that $\mult_x D_2>bN$ for some (fixed) constant $b>0$. To give a more precise statement, we need some definition.

\begin{defn}[c.f. \cite{schwede_adjoint}] 
Let $(X,D)$ be a pair. The test module of $(X,D)$ is defined to be
\[\tau(K_X,D)=\sum_{e\ge0}\mathrm{Tr}^e_X(F_{*}^{e}(\omega_{X}(-\left\lceil p^{e}D\right\rceil ))) \subseteq \omega_X.
\]
\end{defn}

It is clear from our discussion in Section \ref{sec:test_ideal} that over the smooth locus of $X$, the test module coincides with $\tau(X,D)\cdot\omega_X$ where $\tau(X,D)$ is the test ideal of $(X,D)$.

\begin{lem} \label{lem:global generation}
Let $(X,D)$ be a pair and $L$ a Weil divisor on $X$ such that $L-D$ is ample. Let $x\in X$ be a smooth point such that $\epsilon_F(L-D,x)>1$. Then the sheaf $\tau(K_X,D)\otimes\cO_X(L)$ is globally generated at $x$, where $\tau(K_X,D)$ is the test module of the pair $(X,D)$. 
\end{lem}

Here the notation $\epsilon_F(L,x)$ stands for the Frobenius-Seshadri constants \cite{F-seshadri,takumi} of the divisor $L$ at the smooth point $x$.

\begin{proof}
Let $e\gg 0$ be a sufficiently divisible integer and consider the following commutative diagram
\[\xymatrix@C=1em{
0\ar[r] & F_{*}^{e}(\mathfrak{m}_{x}^{[p^{e}]}\cdot\omega_{X}(-\left\lceil p^{e}D\right\rceil ))\ar[r]\ar[d]^{\mathrm{Tr}^{e}} & F_{*}^{e}(\omega_{X}(-\left\lceil p^{e}D\right\rceil ))\ar[r]\ar[d]^{\mathrm{Tr}^{e}} & F_{*}^{e}(\omega_{X}(-\left\lceil p^{e}D\right\rceil )\otimes\mathcal{O}_{X}/\mathfrak{m}_{x}^{[p^{e}]})\ar[r]\ar[d]^{\mathrm{Tr}^{e}} & 0\\
0\ar[r] & \mathfrak{m}_{x}\cdot\tau(K_X,D)\ar[r] & \tau(K_X,D)\ar[r] & \tau(K_X,D)\otimes k_{x}\ar[r] & 0.
}
\]
Since $x$ is a smooth point and $e\gg 0$, the trace map $\mathrm{Tr}^{e}:F_{*}^{e}(\omega_{X}(-\left\lceil p^{e}D\right\rceil ))\rightarrow\tau(X,D)\omega_{X}$ is locally surjective around $x$, hence after tensoring with $L$, we get another commutative diagram
\begin{equation} \label{eq:diagram}
\xymatrix{
F_{*}^{e}(\omega_{X}(p^{e}L-\left\lceil p^{e}D\right\rceil ))\ar[r]\ar[d] & F_{*}^{e}(\omega_{X}(p^{e}L-\left\lceil p^{e}D\right\rceil )\otimes\mathcal{O}_{X}/\mathfrak{m}_{x}^{[p^{e}]})\ar[d]^{\phi}\\
\tau(K_X,D)\otimes\cO_X(L)\ar[r] & \tau(K_X,D)\otimes\cO_X(L)\otimes k_{x}
}
\end{equation}
whose vertical maps are surjective around $x$. In particular, $\phi$ induces a surjection on global sections since both sheaves in question have zero-dimensional support. On the other hand, since $L-D$ is ample and $\epsilon_F(L-D,x)>1$, there exists $m\in\bZ_{\ge0}$ such that $(p^e-m)(L-D)$ is Cartier, $\omega_X(mL-\lceil mD \rceil )$ is globally generated and 
\[H^{0}(X,\cO_X((p^{e}-m)(L-D)))\rightarrow H^{0}(X,\cO_X((p^{e}-m)(L-D))\otimes\mathcal{O}_{X}/\mathfrak{m}_{x}^{[p^{e}]})
\]
is surjective. Tracing through the diagram, it follows that the two horizontal maps in (\ref{eq:diagram}) also induce surjection on global sections. In particular, $\tau(K_X,D)\otimes\cO_X(L)$ is globally generated at $x$.
\end{proof}

To construct the required divisor $D_2$ (and sometimes even $D_3$), we separate into two cases.

First suppose that $\mult_y D_1<N$ for all $y\neq x$ around $x$. Then by Lemma \ref{lem:mult and test ideal}, the test ideal $\tau(X,\frac{1}{N}D_1)\subseteq \mathfrak{m}_x^{N-2}$ and is trivial in a punctured neighbourhood of $x$. By \cite[Proposition 2.12]{F-seshadri}, \[\epsilon_F(-K_X,x)\ge \frac{1}{3}\epsilon(-K_X,x)>\frac{2}{3}.
\]
Thus if $N\gg 0$ and $L=-2K_X$ then $\epsilon_F(L-\frac{1}{N}D_1,x)>1$. Therefore by Lemma \ref{lem:global generation}, $\tau(K_X,\frac{1}{N}D_1)\otimes\cO_X(L)$ is globally generated at $x$. In particular, as $\tau(K_X,\frac{1}{N}D_1)\otimes\cO_X(L)=\tau(X,\frac{1}{N}D_1)\cO_X(-K_X)$ over the smooth locus of $X$, we get divisors $D_i\sim -K_X$ ($i=1,2,3$) such that locally $D_1\cap D_2 \cap D_3$ is supported at $x$ and $\mult_x D_i\ge N-2$. Let $\mathbf{D}=(D_1,D_2,D_3)$. By Lemma \ref{lem:mult and test ideal}, for any $t\in \Delta(\mathbf{D})$ we have $t_i\le \frac{3}{N-2}$, which gives $\Gamma(\mathbf{D})\le -\frac{6}{N-2}K_X<-\frac{1}{2}K_X$ as desired.

Hence we may assume that there is a curve $C\subseteq X$ containing $x$ such that $\mult_y D_1\ge N$ for all $y\in C$ and by Lemma \ref{lem:multiple bounded}, $\mult_z D_1<a$ if $z\not\in C$ (all statements are local around $x$). Replacing $x$ by a general point of $C$ we may also assume that $C$ is smooth at $x$. Let $0<c<\min\{\frac{1}{2},\frac{1}{a}\}$. By Lemma \ref{lem:mult and test ideal} again, the test ideal $\tau(X,cD_1)\subseteq \mathfrak{m}_y^{\lfloor cN \rfloor -2}$ for all $y\in C$ and is trivial outside of $C$. Let $L=-3K_X$, then as in the previous case we have $\epsilon_F(L-cD_1,x)>\epsilon_F(-2K_X,x)>1$, thus $\tau(K_X,cD_1)\otimes\cO_X(-3K_X)$ is globally generated at $x$. In particular, we see that there exists a constant $b>0$ depending only on $\epsilon$ and a divisor $D_2\sim_\bQ -K_X$ whose support does not contain $\Supp(D_1)$ such that $\mult_y(D_i)>bN$ for all $y\in C\subseteq D_1\cap D_2$. 

Let $0<\delta<\frac{1}{2}$ be any rational number such that $\delta(2+\epsilon)>1$. In particular we get $\epsilon(-\delta K_X,x)>1$. By the definition of Seshadri constants, we have $(-\delta K_X\cdot C)>\mult_x C\ge 1$, hence there exists an effective divisor $D_3\sim_\bQ -K_X$ whose support doesn't contain $C$ such that $\mult_x(D_3|_C)>\delta^{-1}$. As before we may assume that $D_3$ is irreducible. Let $\mathbf{D}=(D_1,D_2,D_3)$. Our last claim is

\begin{lem}
$\Gamma(\mathbf{D})\sim_\bQ -\lambda K_X$ where $\lambda\le \frac{6}{N}+\delta$.
\end{lem}

\begin{proof}
Let $\Delta_i$ be the support of $D_i$ and write $D_i=b_i \Delta_i$ ($i=1,2,3$). Let $v=(a_1,a_2,a_3)$ be the dominant vertex of $\Delta(\mathbf{D})$ and $\nu_i(e)$ the $F$-threshold function at $x$ for the divisors $\Delta_i$ as before. By Lemma \ref{lem:mult and test ideal} it is clear that $a_1,a_2\le \frac{3}{N}$. We claim that $a_3\le \delta$. Let $W$ be the scheme-theoretic intersection of $\Delta_1$ and $\Delta_2$. Since $W$ is supported at $C$ (at least locally around $x$), there exists an integer $r>0$ such that $\cI_C^{[p^r]}\subseteq \cI_W$. For each $e\in \bZ_{>0}$, let $m(e)$ be the smallest integer such that $m(e)\cdot \mult_x(\Delta_3|_C)\ge p^e$. If $(f=0)$ is the local defining equation of $\Delta_3$, then we have
\[f^{m(e)}\subseteq \mathfrak{m}_x^{p^e} + \cI_C = \mathfrak{m}_x^{[p^e]} + \cI_C 
\]
where the second equality holds since $\mathfrak{m}_{C,x}^p=\mathfrak{m}_{C,x}^{[p]}$ on the smooth curve $C$. It follows that
\begin{equation} \label{eq:Delta_3}
f^{p^r m(e)}\subseteq (\mathfrak{m}_x^{[p^e]})^{[p^r]} + \cI_C^{[p^r]} \subseteq \mathfrak{m}_x^{[p^{e+r}]}+ \cI_W.
\end{equation}
It is clear that
\[\mathrm{Tr}^{e}(F_{*}^{e}(\mathcal{O}_{X}((1-p^{e})(K_{X}+B)-\sum_{i=1}^2 t_i(e)\Delta_i)\cdot \mathfrak{m}_x^{[p^e]}))\subseteq\mathfrak{m}_x,
\]
hence combining with (\ref{eq:not_surj}) and (\ref{eq:Delta_3}) we get
\[\mathrm{Tr}^{e+r}(F_{*}^{e+r}(\mathcal{O}_{X}((1-p^{e+r})(K_{X}+B)-\sum_{i=1}^2 t_i(e+r)\Delta_i-p^r m(e))\Delta_3))\subseteq\mathfrak{m}_x.
\]
Therefore, $\nu_3(e+r)\le p^r m(e)$ and as $a_3=\lim_{e\rightarrow\infty} \frac{\nu_3(e+r)}{b_3(p^{e+r}-1)}$, we deduce that 
\[a_3\le \frac{1}{\mult_x (D_3|_C)}<\delta
\]
as claimed. The lemma now follows as by our construction $D_i\sim_\bQ -K_X$ for each $i$.
\end{proof}

Summing up, we eventually have

\begin{thm} \label{thm:rho=1}
The set of $\bQ$-factorial terminal Fano threefolds $X$ of Picard number one such that $\epsilon(-K_X,x)>2+\epsilon$ for some $x$ is weakly bounded.
\end{thm}

\begin{proof}
Let $0<\delta<\frac{1}{2}$ be chosen as before, then by the previous lemma, as $N\gg 0$ there exists $\mathbf{D}=(D_1,D_2,D_3)$ such that $\cap D_i$ is locally supported at a very general point $x$ and $\Gamma(\mathbf{D})\sim_\bQ -\lambda K_X$ for some $\lambda<\frac{1}{2}$. This contradicts Lemma \ref{lem:weak connectedness-2}.
\end{proof}

Finally we finish the proof of the main theorem and its corollaries.

\begin{proof}[Proof of Theorem \ref{thm:weakbdd_m}]
This follows from Lemma \ref{lem:reduction}, Corollary \ref{cor:mfs}, Lemma \ref{lem:3-fold-mfs} and Theorem \ref{thm:rho=1}.
\end{proof}

\begin{proof}[Proof of Theorem \ref{thm:weakbdd}]
By \cite[Theorem 6.4]{demailly} we have $\epsilon(-K_X,x)=\epsilon_m(-K_X,x)$ for all smooth point $x\in X$, so the result follows immediately from Theorem \ref{thm:weakbdd_m}.
\end{proof}

\begin{proof}[Proof of Corollary \ref{cor:birbdd}]
By \cite[Theorem 1.1]{F-seshadri}, $|-2K_X|$ induces a birational map, so the corollary follows from Theorem \ref{thm:weakbdd} and \cite[Lemma 2.4.2]{hmx}.
\end{proof}

\bibliography{ref}
\bibliographystyle{alpha}

\end{document}